\documentclass[a4paper,11pt]{article}
\usepackage{amssymb,amsmath,amsthm}
\usepackage{fullpage}
\usepackage{hyperref}
\usepackage{graphicx}
\usepackage{stackrel}
\usepackage[all]{xy}
\newcommand{\ie}{\emph{i.e.}}
\newcommand{\eg}{\emph{e.g.}}
\newcommand{\cf}{\emph{cf.}}

\newcommand{\Real}{\mathbb{R}}
\newcommand{\Nat}{\mathbb{N}}

\newcommand{\supp}{\mathop{\mathrm{supp}}\nolimits}

\newcommand{\Euler}{\mathop{\mathrm{e}}\nolimits}
\newcommand{\Dom}{\mathsf{D}}

\newcommand{\Hilbert}{\mathcal{H}}
\newcommand{\eps}{\varepsilon}
\newcommand{\sii}{L^2}
\newcommand{\der}{\mathrm{d}}
\newtheorem{Theorem}{Theorem}
\newtheorem{Corollary}{Corollary}
\newtheorem{Proposition}{Proposition}
\newtheorem{Lemma}{Lemma}
\newtheorem{Conjecture}{Conjecture}
\theoremstyle{definition}

\begin{document}
%
\title{\textbf{\Large The Hardy inequality
and the heat flow in curved wedges}}
\author{David Krej\v{c}i\v{r}{\'\i}k}
\date{\small 
\emph{Department of Theoretical Physics,
Nuclear Physics Institute ASCR, \\
25068 \v{R}e\v{z}, Czech Republic; 
krejcirik@ujf.cas.cz}
\medskip \\
11 January 2016}
\maketitle
\begin{abstract}
\noindent
We show that the polynomial decay rate of the heat semigroup 
of the Dirichlet Laplacian in curved planar wedges 
that are obtained as a compactly supported perturbation of straight wedges
equals the sum of the usual dimensional decay rate 
and a multiple of the reciprocal value of the opening angle.
To prove the result, we develop the method of self-similar variables 
for the associated heat equation and study the asymptotic behaviour 
of the transformed non-autonomous parabolic problem for large times.
We also establish an improved Hardy inequality for the Dirichlet Laplacian
in non-trivially curved wedges and state a conjecture about an improved 
decay rate in this case.
\end{abstract}
%

\section{Introduction}
%
Relations between the geometry of a domain,
spectral properties of an associated differential operator
and time evolution of the generated heat flow
are one of the vintage problems of mathematics.
The interest lies in the involvement of different fields of mathematics
(notably differential geometry, spectral theory and partial differential equations)
and a wide scope of applications 
(from classical dissipative systems 
and stochastic analysis to quantum mechanics).
We refer to the classical book of Davies~\cite{Davies_1989}
and the recent monograph of Grigor'yan~\cite{Grigoryan} 
with many references.

In this paper, we are interested in a large time behaviour 
of the heat semigroup associated with the Dirichlet Laplacian
in the geometric setting of two-dimensional unbounded domains 
which are obtained as a curved deformation of 
the \emph{straight wedge}
\begin{equation}\label{wedge0}
  \Omega_0 := \Big\{
  \big(
  r \cos\varphi, r \sin\varphi
  \big)
  \ \Big| \
  r \in (0,\infty) 
  \,, \ 
  \varphi \in (0,2\pi a)
  \Big\}
\end{equation}
with $a \in (0,1]$, see Figure~\ref{Fig}.
The Brownian motion in straight cones has been extensively studied; 
see in particular \cite{DeBlassie_1987}, \cite{Banuelos-Smits_1997}
and the recent review~\cite{Denisov-Wachtel_2015} with many further references.
However, the present curved feature of the boundary seems to be a new aspect.
Our main result about the heat semigroup (Theorem~\ref{Thm.main})
says that the polynomial 
decay rate is insensitive to our specific
compactly supported perturbations of the straight wedges.
Although the result is perhaps heuristically expectable, 
it is still non-trivial
because there is no general theory of properties of the heat semigroup
under compactly supported perturbations.
Moreover, the result admits an interesting stochastic interpretation
in terms of properties of the Brownian motion.

It is well known that the large time behaviour of the heat kernel
is related to transient/re\-cur\-rent properties of the Brownian motion
and spectral-threshold characteristics of the generator;
see, \eg, Pinsky's monograph~\cite{Pinsky} 
and a recent survey of Pinchover~\cite{Pinchover_2013}.
One way how to characterise the latter in our context
is through the existence/non-existence of the Hardy inequality
for the Dirichlet Laplacian.
There is an extensive literature on Hardy inequalities in conical domains
(see~\cite{Devyver-Pinchover-Psaradakis} and references therein),
but curved wedges do not seem to be considered in the present context.
Using a classical Hardy inequality for simply connected domains 
with a Hardy weight expressed by the distance to the boundary
(see, \eg, \cite{Ancona_1986}, \cite{Davies_1995}, 
\cite{Davies_1999}, \cite{Laptev-Sobolev_2008}),
one can instantly obtain a number of Hardy inequalities in our curved wedges.
Using curvilinear coordinates, it is also possible
to get a Hardy inequality
with a Hardy weight expressed by the distance to the conical singularity
which becomes optimal for straight wedges.
In this paper, we go beyond these immediate inequalities
and establish an improved Hardy inequality which holds
if, and only if, the wedge is curved (Theorem~\ref{Thm.Hardy}).

To state the main results of this paper, 
let us first introduce a general class of curved wedges.  
Given a function $\theta:(0,\infty)\to\Real$
and a number $a \in (0,1]$, let us consider 
the two-dimensional domain
\begin{equation}\label{wedge}
  \Omega := \Big\{
  \big(
  r \cos[\varphi+\theta(r)], r \sin[\varphi+\theta(r)]
  \big)
  \ \Big| \
  r \in (0,\infty) 
  \,, \ 
  \varphi \in (0,2\pi a)
  \Big\}
  \,.
\end{equation}
We call~$\Omega$ a \emph{curved wedge} of \emph{opening angle}~$2\pi a$.
We are primarily interested in the large time behaviour of the heat semigroup
\begin{equation}\label{semigroup}
  \Euler^{t \Delta_D^{\Omega}}
  \,,
\end{equation}
where $-\Delta_D^{\Omega}$ 
denotes the (non-negative) Dirichlet Laplacian in $\sii(\Omega)$.

The geometry of a curved wedge can be quite complex, see Figure~\ref{Fig}.
In fact, any type of unbounded domain from Glazman's classification
\cite[Sec.~49]{Glazman} (see also \cite[Thm.~X.6.1]{Edmunds-Evans})
can be realised: 
quasi-conical (\ie~containing arbitrarily large disks), 
quasi-cylindrical (\ie~not quasi-conical but containing
a sequence of identical pairwise disjoint disks)
and quasi-bounded (neither quasi-conical nor quasi-cylindrical).
A characteristic assumption of this paper is that the derivative~$\theta'$
vanishes at infinity, 
so that~$\Omega$ is a local perturbation of the straight wedge~\eqref{wedge0}.
More specifically, we assume that $\theta \in C^1((0,\infty))$ 
is such that
\begin{equation}\label{local}
  \lim_{r \to \infty} \theta'(r) = 0
  \,.
\end{equation}
For a majority of the results in this paper,
we actually require that~$\theta'$ decays at infinity fast enough (\cf~\eqref{Ass})
or even that $\theta'(r)=0$ for all sufficiently large~$r$. 
Then~$\Omega$ is a compactly supported perturbation of the straight wedge~$\Omega_0$.
In particular, $\Omega$~shares the property of~$\Omega_0$
being quasi-conical, 
and consequently (see, \eg, \cite[Thm.~X.6.5]{Edmunds-Evans})
\begin{equation}\label{spectrum}
  \sigma(-\Delta_D^{\Omega})
  = \sigma_\mathrm{ess}(-\Delta_D^{\Omega})
  = [0,\infty)
  \,.
\end{equation}
It follows from the spectral mapping theorem that 
$
  \big\|\Euler^{t \Delta_D^{\Omega}}\big\|_{\sii(\Omega) \to \sii(\Omega)} = 1
$
for all $t \geq 0$.

To reveal a decay of the heat semigroup in a more adapted topology, 
we introduce a weighted space
\begin{equation}\label{weighted}
  \sii_w(\Omega) := \sii(\Omega,w(x)\,\der x)
  \,, \qquad \mbox{where} \qquad
  w(x) := \Euler^{|x|^2/4}
  \,,
\end{equation}
and reconsider~\eqref{semigroup} as an operator 
from $\sii_w(\Omega) \subset \sii(\Omega)$ to~$\sii(\Omega)$.
As a measure of the decay of the heat semigroup,
we then consider the polynomial \emph{decay rate}
\begin{equation}\label{rate}
  \Gamma_{\theta,a}
  := \sup \Big\{ \gamma \left| \
  \exists C_\gamma > 0, \, \forall t \geq 0, \
  \big\|\Euler^{t \Delta_D^\Omega}\big\|_{
  \sii_w(\Omega)
  \to
  \sii(\Omega)
  }
  \leq C_\gamma \, (1+t)^{-\gamma}
  \Big\} \right.
  .
\end{equation}

Our main result reads as follows.
\begin{Theorem}\label{Thm.main}
Let $\theta \in C^1((0,\infty))$ be such that 
\begin{equation}\label{Ass}
  r \mapsto r\,\theta'(r) \in L^\infty((0,\infty))
\end{equation}
and $\supp\theta'$ is compact in~$\Real$.
Then
\begin{equation}\label{result}
  \Gamma_{\theta,a} = \frac{1}{2}+ \frac{1}{4a} 
  \,.
\end{equation}
\end{Theorem}

We remark that~\eqref{Ass} implies the asymptotic 
behaviour at infinity~\eqref{local},
while singularities of~$\theta'$ at zero are allowed by this hypothesis. 

The~$1/2$ on the right hand side of~\eqref{result}
is the usual power in the polynomial decay rate 
of the heat semigroup in~$\Real^2$
(more generally, one has~$d/4$ in~$\Real^d$),
while the additional $a$-dependent term is an improvement
due to the extra Dirichlet boundary conditions.
The results agree with the decay rates for the straight wedges, 
where the heat kernel can be written down explicitly 
by a separation of variables
(see~\cite[p.~379]{Carslaw-Jaeger} 
or more generally~\cite[Lem.~1]{Banuelos-Smits_1997}).
Despite of the existence of an extensive literature on properties 
of the Brownian motion in straight cones
(see in particular \cite{DeBlassie_1987}, \cite{Banuelos-Smits_1997}
and the recent review~\cite{Denisov-Wachtel_2015} with many further references),
the present Theorem~\ref{Thm.main} seems to be new 
because of the curved feature of the boundary of~$\Omega$.

We expect that the conclusion of Theorem~\ref{Thm.main}
remains valid even if~$\theta'$ is not compactly supported
but it decays to zero sufficiently fast at infinity.
On the other hand, 
it is possible that the decay of the heat semigroup
is faster than polynomial provided that~$\theta'$ 
decays to zero very slowly at infinity.
This conjecture is supported by subexponential asymptotics of the heat kernel 
of Schr\"odinger operators with slowly decreasing potentials 
\cite{Yafaev_1982}
and related properties of the Brownian motion in parabolic domains
\cite{Banuelos-DeBlassie-Smits_2001}, \cite{Lifshits-Shi_2002}, 
\cite{Li_2003}, \cite{Berg_2003}.

The statement of Theorem~\ref{Thm.main} for solutions~$u$ of
the heat equation 
\begin{equation}\label{Cauchy}
\left\{
\begin{aligned}
  \partial_t u -\Delta_D^\Omega \;\! u &= 0 \,,
  \\
  u(0) &= u_0 \,,
\end{aligned}
\right.
\end{equation}
can be reformulated as follows.
If $u_0 \in \sii_w(\Omega)$, then, for every $\delta>0$, 
there exists a positive constant~$C_\delta$ 
(depending in addition to~$\delta$ on~$a$)
such that
\begin{equation}\label{solution.rate}
  \|u(t)\|_{\sii(\Omega)}
  \leq C_\delta \, (1+t)^{-\Gamma_{\theta,a}+\delta} \,
  \|u_0\|_{\sii_w(\Omega)}
\end{equation}
for each time $t \geq 0$ and any initial datum $u_0 \in \sii_w(\Omega)$.
The constant~$C_\delta$ can in principle explodes as $\delta \to 0$,
but we expect that it can be actually made independent of~$\delta$ in the limit,
relying on other situations where the method of proof that we use 
is known to give optimal decay rates (see below).

From~\eqref{solution.rate} it is possible to deduce
the following pointwise bound.
\begin{Corollary}\label{Thm.point}
Let $\theta \in C^1((0,\infty))$ be such that~\eqref{Ass} holds  
and $\supp\theta'$ is compact in~$\Real$.
For any positive number~$\delta$, 
there exists a constant $\tilde{C}_{\delta}$
such that the solution~$u$ of~\eqref{Cauchy} 
with an arbitrary initial datum $u_0 \in \sii_w(\Omega)$ obeys
\begin{equation}\label{result.point}
  \forall t \geq 1 \,, \qquad
  \|u(t)\|_{L^\infty(\Omega)}
  \leq 
  \tilde{C}_{\delta} \, t^{-\Gamma_{\theta,a}+\delta} \,
  \|u_0\|_{\sii_w(\Omega)} 
  \,,
\end{equation}
where~$\Gamma_{\theta,a}$ is given by~\eqref{result}.
\end{Corollary}

Our proof of Theorem~\ref{Thm.main} is based on the method
of self-similar variables that was originally developed 
for the heat equation in the whole Euclidean space 
by Escobedo and Kavian in~\cite{Escobedo-Kavian_1987} 
and subsequently applied to the heat equation with
variable coefficients in numerous works;
we refer to the recent paper~\cite{CK} for an extensive reference list.
The method was also applied to the heat equation 
in non-trivial geometries, 
namely in twisted tubes in~\cite{KZ1} and~\cite{KZ2}
and in curved manifolds in~\cite{KKolb}. 
We shall see that the present problem exhibits certain
similarities with the problem in twisted tubes 
when suitable curvilinear coordinates are applied,
\cf~Section~\ref{Sec.curve}.

Theorem~\ref{Thm.main} and its Corollary~\ref{Thm.point}
are proved in the following Section~\ref{Sec.proof}. 
Theorem~\ref{Thm.main} follows as a consequence 
of lower (Theorem~\ref{Thm.main.lbound}) 
and upper (Theorem~\ref{Thm.optimal}) bounds to the decay rate.
Corollary~\ref{Thm.point} is a special case of 
a more general result (Theorem~\ref{Thm.point.bis}).

To the end of the paper we append Section~\ref{Sec.end},
where we raise a conjecture about an improved 
(possibly non-polynomial) decay rate
in non-trivially curved wedges (\ie~$\theta'\not=0$)
with respect to the straight wedges.
In fact, although the spectrum~\eqref{spectrum} 
is insensitive to variations of the boundary of~$\Omega$
provided that~\eqref{local} holds, 	
there is an improved Hardy inequality  
in the former case:

\begin{Theorem}\label{Thm.Hardy}
Let $\theta \in C^1([0,\infty))$ 
be such that $\supp\theta'$ is compact in~$\Real$. 
If $\theta' \not= 0$,
then there exists a positive constant~$c$ such that  
\begin{equation}\label{Hardy}
  \forall  u \in H_0^1(\Omega)
  \,, \qquad
  \int_{\Omega} |\nabla u(x)|^2 \, \der x
  - \frac{1}{4 a^2} 
  \int_{\Omega} \frac{|u(x)|^2}{|x|^2} \, \der x 
  \geq 
  c \int_{\Omega} \frac{|u(x)|^2}{1+|x|^2\log^2(|x|)} \, \der x 
  \,.
\end{equation}
\end{Theorem}
Inequality~\eqref{Hardy} holds in straight wedges 
(\ie~$\theta'$ vanishes identically) with $c=0$
and it is optimal in this case
(in the sense that it cannot be improved by ad Ing 
a positive term on the right hand side).
The message of Theorem~\ref{Thm.Hardy} is that 
a better inequality holds whenever the wedge is non-trivially curved.
Although there is an extensive literature on Hardy inequalities
in conical domains
(see~\cite{Devyver-Pinchover-Psaradakis} and references therein), 
Theorem~\ref{Thm.Hardy} seems to be new.

\section{The polynomial decay rate of the heat flow}\label{Sec.proof}
%
This section is devoted to proofs of Theorem~\ref{Thm.main}
and its Corollary~\ref{Thm.point}.

\subsection{Curvilinear coordinates}\label{Sec.curve}
%
As usual, we understand~$-\Delta_D^\Omega$ as the self-adjoint
operator in $\sii(\Omega)$ associated with the quadratic form
$Q_D^\Omega[u]:=\|\nabla u\|_{\sii(\Omega)}^2$, 
$\Dom(Q_D^\Omega) := H_0^1(\Omega)$.
In this subsection we express the Dirichlet Laplacian
and the associated heat equation~\eqref{Cauchy}
in natural curvilinear coordinates.

By definition~\eqref{wedge}, 
$\Omega$~coincides with the image of the mapping
$\mathcal{L} : U \to \Real^2$, where
\begin{equation}\label{layer}
  \mathcal{L}(r,\varphi) := 
  \big(
  r \cos[\varphi+\theta(r)], r \sin[\varphi+\theta(r)]
  \big)
\end{equation}
and $U :=  (0,\infty) \times (0,2\pi a)$.
Let $\theta \in C^1((0,\infty))$.
Then it is easy to see that $\mathcal{L}$ induces a $C^1$-smooth diffeomorphism
between~$U$ and~$\Omega$.
The corresponding metric 
$G := \nabla\mathcal{L} \cdot (\nabla\mathcal{L})^T$
acquires the form
\begin{equation}\label{metric}
  G(r,\varphi) = 
  \begin{pmatrix}
    1 + r^2 \theta'(r)^2 & r^2 \theta'(r) 
    \\
    r^2 \theta'(r) & r^2 
  \end{pmatrix} 
  , \qquad
  \det\big(G(r,\varphi)\big) = r^2
  \,.
\end{equation}

Introducing also $\mathcal{L}_0 : U \to \Real^2$ by
$\mathcal{L}_0(r,\varphi) := (r\cos\varphi,r\sin\varphi) \in \Omega_0$,
we may understand~\eqref{layer} as an identification 
of~$\Omega$ with the straight wedge~$\Omega_0$
introduced in~\eqref{wedge0} via the commutative diagram 
\begin{equation}\label{diagram}
\xymatrix{
   &U  \ar[dl]_{\mathcal{L}_0} \ar[dr]^{\mathcal{L}} &
   \\	
   \Omega_0 \ar[rr]_{\mathcal{L}\circ\mathcal{L}_0^{-1}} && \Omega \,.
}
\end{equation}

Using the unitary transform 
\begin{equation}\label{unitary}
  \mathcal{U}: 
  \sii(\Omega) \to \sii(U,r \, \der r \, \der\varphi )
  =: \Hilbert
\end{equation}
defined by $u \mapsto u \circ \mathcal{L}$,
we consider the unitarily equivalent operator
$H := \mathcal{U}(-\Delta_D^\Omega)\mathcal{U}^{-1}$.
The latter is just the operator in the new Hilbert space~$\Hilbert$
associated with the transformed form
$h[\psi] := Q_D^\Omega[\mathcal{U}^{-1}\psi]$, 
$\Dom(h) := \mathcal{U}\Dom(Q_D^\Omega)$.
For later purposes, we also introduce the unitary transform
$
  \mathcal{U}_0: 
  \sii(\Omega_0) \to \sii(U,r \, \der r \, \der\varphi )
$
by $u \mapsto u \circ \mathcal{L}_0$.

\begin{Proposition}\label{Prop.domain}
Let $\theta \in C^1((0,\infty))$ be such that~\eqref{Ass} holds.
Then 
\begin{align}
  h[\psi] &= \int_U 
  \left[
  \left|\big(\partial_r-\theta'(r)\partial_\varphi\big)\psi\right|^2
  + \left|\frac{\partial_\varphi \psi}{r}\right|^2
  \right]
  r \, \der r \, \der\varphi 
  \,,
  \label{h-form}
  \\
  \Dom(h) &= \Hilbert_0^1 := \overline{C_0^1(U)}^{\|\cdot\|_{\Hilbert^1}}
  \,,
  \label{h-form.domain}
\end{align}
where 
$$
  \|\psi\|_{\Hilbert^1} := \sqrt{\int_U 
  \left[
  \left|\partial_r\psi\right|^2
  + \left|\frac{\partial_\varphi \psi}{r}\right|^2
  + |\psi|^2
  \right]
  r \, \der r \, \der\varphi}
  \,.
$$
\end{Proposition}
\begin{proof}
If $u \in C_0^1(\Omega)$, 
then $\psi := u \circ \mathcal{L} \in C_0^1(U)$
and~\eqref{metric} yields that~$h$ acts as in~\eqref{h-form}.
It remains to show that the norm induced by~$h$ 
is equivalent to $\|\cdot\|_{\Hilbert^1}$.
Note that the latter is just the norm of $H^1(\Omega_0)$ 
written in polar coordinates.
Assumption~\eqref{Ass} means that there exists a constant~$C$
such that $|\theta'(r)| \leq C/r$ for all $r \in (0,\infty)$.
Elementary estimates yield
\begin{equation}\label{elementary1}
\begin{aligned}
  h[\psi] 
  & \geq \int_U 
  \left[
  \epsilon \, |\partial_r\psi|^2
  - \frac{\epsilon}{1-\epsilon} \,
  \left|\theta'(r)\partial_\varphi\psi\right|^2
  + \left|\frac{\partial_\varphi \psi}{r}\right|^2
  \right]
  r \, \der r \, \der\varphi 
  \\
  & \geq \int_U 
  \left[
  \epsilon \, |\partial_r\psi|^2
  + \left(1-\frac{C^2 \, \epsilon}{1-\epsilon}\right)
  \left|\frac{\partial_\varphi \psi}{r}\right|^2
  \right]
  r \, \der r \, \der\varphi 
\end{aligned}
\end{equation}
and
\begin{equation}\label{elementary2}
  h[\psi] 
  \leq \int_U 
  \left[
  2 \, |\partial_r\psi|^2
  + \left(1+2 \, C^2 \right)
  \left|\frac{\partial_\varphi \psi}{r}\right|^2
  \right]
  r \, \der r \, \der\varphi
\end{equation}
for all $\psi \in C_0^1(U)$ and every $\epsilon \in (0,1)$.
Hence, the equivalence of the norms follows 
by choosing~$\epsilon$ sufficiently small.
\end{proof}

In a distributional sense we may write
\begin{equation}\label{op.twist}
  H = -\frac{1}{r} 
  \big(\partial_r-\theta'(r)\partial_\varphi\big) r 
  \big(\partial_r-\theta'(r)\partial_\varphi\big) 
  - \frac{1}{r^2} \, \partial_\varphi^2 
  \,.
\end{equation}
We notice that~$H$ has a structure similar to the Dirichlet Laplacian
in a twisted tube when expressed in suitable curvilinear coordinates,
\cf~\cite{EKK}.

By the unitary equivalence above, 
it is enough to establish the result~\eqref{result}
for the heat semigroup $\Euler^{-t H}$ in~$\Hilbert$. 
Given $\psi_0 \in \Hilbert$, $\psi(t) := \Euler^{-t H} \psi_0$
is a solution of the Cauchy problem
\begin{equation}\label{Cauchy.straight}
\left\{
\begin{aligned}
  \partial_t \psi + H \psi &= 0 \,,
  \\
  \psi(0) &= \psi_0 \,.
\end{aligned}
\right.
\end{equation}
By the Hille-Yosida theorem~\cite[Thm.~7.7]{Brezis_new},
\begin{equation}\label{sol.smooth}
  \psi\in C^0\big([0, \infty); \Hilbert\big)
  \cap C^1\big((0, \infty); \Hilbert\big)
  \cap C^0\big((0,\infty); \Dom(H)\big)
  \,.
\end{equation}

By the Beurling-Deny criterion, 
$\Euler^{-t H}$ is positivity-preserving for all $t \geq 0$.
Moreover, the real and imaginary parts
of the solution~$\psi$ of~\eqref{Cauchy.straight} evolve separately.
By writing $\psi = \Re\psi + i \, \Im\psi$ and solving~\eqref{Cauchy.straight}
with initial data $\Re\psi_0$ and $\Im\psi_0$,
we may therefore reduce the problem to the case of a real function~$\psi_0$,
without restriction.
Consequently, all the functional spaces are considered to be real 
in this section.

\subsection{Self-similar variables}
%
Now we adapt the method of self-similar variables. 
It was originally developed 
for the heat equation in the whole Euclidean space in~\cite{Escobedo-Kavian_1987}. 
We refer to~\cite{KZ1} for an application to twisted tubes 
which exhibit technical similarities with the present geometric setting.

If $(r,\varphi,t) \in U \times (0,\infty)$ are the initial	
space-time variables for the heat equation~\eqref{Cauchy.straight},
we introduce \emph{self-similar variables}
$(\rho,\varphi,s) \in U \times (0,\infty)$ by
\begin{equation}\label{relationship}
  \rho := (t+1)^{-1/2} \, r
  \,, \qquad
  s := \log(t+1)
  \,.
\end{equation}
The angular variable~$\varphi$ is not changed by this transformation.
We naturally write
$  
  y := (\rho\cos(\varphi),\rho\sin\varphi) \in \Omega_0
$, 
so that $|y|^2=\rho^2$.

If~$\psi$ is a solution of~\eqref{Cauchy.straight},
we then define a new function
\begin{equation}\label{SST}
  \tilde{\psi}(\rho,\varphi, s) :=
  \Euler^{s/2} \, \psi\big(\Euler^{s/2}\rho,\varphi,\Euler^s-1\big)
  \,.
\end{equation}
The inverse transform is given by
\begin{equation}\label{ISST}
  \psi(r,\varphi,t)=(t+1)^{-1/2} \, 
  \tilde{\psi}\big((t+1)^{-1/2} r,\varphi,\log(t+1)\big)
  \,.
\end{equation}
It is straightforward to check that~$\tilde{\psi}$
satisfies a weak formulation of the non-autonomous parabolic problem
\begin{equation}\label{Cauchy.ss}
\left\{
\begin{aligned}
  \partial_s\tilde{\psi} -\frac{1}{2} \, \rho \, \partial_\rho \tilde{\psi}
  - \frac{1}{2} \, \tilde{\psi}
  + H_s \tilde{\psi} &= 0 \,,
  \\
  \tilde{\psi}(0) &= \psi_0 \,,
\end{aligned}
\right.
\end{equation}
where
\begin{equation}\label{op.self}
  H_s := -\frac{1}{\rho} 
  \big(\partial_\rho-\theta_s'(\rho)\partial_\varphi\big) \rho 
  \big(\partial_\rho-\theta_s'(\rho)\partial_\varphi\big) 
  - \frac{1}{\rho^2} \, \partial_\varphi^2 
\end{equation}
with the rescaled function
\begin{equation}\label{scaled}
  \theta_s'(\rho) := e^{s/2} \, \theta'(e^{s/2}\rho)
  \,.
\end{equation}

The self-similarity transform $\psi \mapsto \tilde{\psi}$
acts as a unitary transform in 
$\Hilbert \equiv \sii(U,\rho \, \der \rho \, \der\varphi)$;
indeed, we have
\begin{equation}\label{preserve}
  \|\psi(t)\|_{\Hilbert} = \|\tilde{\psi}(s)\|_{\Hilbert}
\end{equation}
for all $s,t \in (0,\infty)$.
This means that we can analyse the asymptotic time behaviour
of the former by studying the latter.
However, the natural space to study the evolution~\eqref{Cauchy.ss}
is not~$\Hilbert$ but rather the transformed analogue of~\eqref{weighted}
\begin{equation}\label{analogue.weighted}
  \Hilbert_w := 
  \sii(U,w(y) \,\rho \, \der \rho \, \der\varphi)
  \,.
\end{equation}
To avoid working in weighted Sobolev spaces,
we proceed equivalently by introducing an additional transform
\begin{equation}\label{noweight}
  \phi(\rho,\varphi,s) := w(y)^{1/2} \, \tilde{\psi}(\rho,\varphi,s) 
  \,.
\end{equation}
Then the Cauchy problem~\eqref{Cauchy.ss} is transformed to
\begin{equation}\label{Cauchy.noweight}
\left\{
\begin{aligned}
  \partial_s\phi 
  + L_s \phi + M_s \phi &= 0 \,,
  \\
  \phi(0) &= \phi_0 := w^{1/2}\psi_0 \,,
\end{aligned}
\right.
\end{equation}
where
\begin{equation}
\begin{aligned}
  L_s &:= -\frac{1}{\rho} 
  \big(\partial_\rho-\theta_s'(\rho)\partial_\varphi\big) \rho 
  \big(\partial_\rho-\theta_s'(\rho)\partial_\varphi\big) 
  - \frac{1}{\rho^2} \, \partial_\varphi^2 
  + \frac{\rho^2}{16}
  \,,
  \\
  M_s &:= -\frac{1}{2} \, \rho \, \theta_s'(\rho) \, \partial_\varphi
  \,. 
\end{aligned}
\end{equation}

More precisely, $L_s$~is defined as the (self-adjoint) operator in~$\Hilbert$ 
associated with the quadratic form
$$
\begin{aligned}
  l_s[\phi] &:= \int_U 
  \left[
  \left|\big(\partial_\rho-\theta_s'(\rho)\partial_\varphi\big)\phi\right|^2
  + \left|\frac{\partial_\varphi \phi}{\rho}\right|^2
  + \frac{\rho^2}{16} \, |\phi|^2
  \right]
  \rho \, \der \rho \, \der\varphi 
  \\
  \Dom(l_s) &:= \overline{C_0^1(U)}^{\|\cdot\|_{l_s}}
  \,,
\end{aligned}
$$ 
where 
$$
  \|\phi\|_{l_s} := \sqrt{l_s[\phi]+\|\phi\|_{\Hilbert}^2}
  \,.
$$
The operator~$M_s$ can be handled as 
a small (non-self-adjoint) perturbation of~$L_s$. 
Indeed, by the Schwarz inequality, we have
$$
  |(\phi,M_s\phi)_\Hilbert| 
  \leq \frac{1}{2} \, C \Euler^{-s/2} 
  \|\rho\phi\|_\Hilbert
  \left\|\frac{\partial_\varphi\phi}{\rho}\right\|_\Hilbert
$$
for all $\phi \in C_0^1(U)$,
where $C := \sup_{r \in (0,\infty)} |r\,\theta'(r)|$
is finite due to~\eqref{Ass}.
Consequently, $M_s$~with the form domain~$\Dom(l_s)$
is relatively form-bounded with respect to~$L_s$
with the relative bound $C \Euler^{-s/2}$. 
Moreover, by an integration by parts, we have
\begin{equation}\label{nsa}
  \Re(\phi,M_s\phi)_\Hilbert 
  = 0
\end{equation}
for all $\phi \in C_0^1(U)$,
and this identity extends to the real part of the quadratic form 
associated with~$M_s$ for all $\phi \in \Dom(l_s)$.

We remark that the form domain $\Dom(l_s)$ as a set is independent of~$s$.
To see it, we compare~$l$ with the following $s$-independent form  
\begin{equation}\label{l-form}
\begin{aligned}
  l[\phi] &:= \int_U 
  \left[
  \left|\partial_\rho\phi\right|^2
  + \left|\frac{\partial_\varphi \phi}{\rho}\right|^2
  + \frac{\rho^2}{16} \, |\phi|^2
  \right]
  \rho \, \der \rho \, \der\varphi 
  \,,
  \\
  \Dom(l) &:= \left\{
  \psi \in \Hilbert_0^1 \ : \ |y|\phi \in \Hilbert
  \right\}
  \,,
\end{aligned}
\end{equation}
where~$\Hilbert_0^1$ is introduced in~\eqref{h-form.domain}.

\begin{Proposition}\label{Prop.equivalent}
Let $\theta \in C^1((0,\infty))$ be such that~\eqref{Ass} holds.
Then there is a positive constant~$C$ such that
\begin{equation}\label{equivalent}
  C^{-1} \, l[\phi] \leq l_s[\phi] \leq C \, l[\phi]
\end{equation}
for every $\phi \in C_0^1(U)$.
\end{Proposition}
\begin{proof}
The inequalities can be obtained in the same way as
in the proof of Proposition~\ref{Prop.domain}.
It is only important to point out here that the constant~$C$
can be chosen independent of~$s$ 
as a consequence of the scaling~\eqref{relationship} 
and the assumption~\eqref{Ass}.
\end{proof}

Consequently, the norms $\|\cdot\|_{l_s}$ and $\|\cdot\|_{l}$ are equivalent.
In particular, $\Dom(l_s)=\Dom(l)$ for all $s \geq 0$.
We also remark that $\Dom(l)$ is compactly embedded in~$\Hilbert$,
which implies that~$L_s$ is an operator with compact resolvent
for all $s \geq 0$.

The fact that~\eqref{Cauchy.noweight} is well posed 
in the scale of Hilbert spaces
$$
  \Dom(l) \subset \Hilbert \subset \Dom(l)^*
$$
follows by an abstract theorem of J.~L.~Lions~\cite[Thm.~10.9]{Brezis_new}
about weak solutions of parabolic problems with time-dependent coefficients.
We refer to~\cite{KZ1} for more details in an analogous situation.
Here it is important that~\eqref{nsa} holds true,
so that the form 
associated with the form sum $L_s\dot{+}M_s$ is bounded and coercive on~$\Dom(l)$.

\subsection{Reduction to a spectral problem}
%
Using~\eqref{nsa},
it follows from~\eqref{Cauchy.noweight} that the identity
\begin{equation}\label{energy}
  \frac{1}{2} \frac{\der}{\der s} \|\phi(s)\|_{\Hilbert}^2 
  = - l_s[\phi(s)] 
\end{equation}
holds for every $s \geq 0$.
Now, as usual for energy estimates, 
we replace the right hand side of~\eqref{energy}
by the spectral bound
\begin{equation}\label{spec.bound}
  l_s[\phi(s)] \geq \lambda_{\theta,a}(s) \, \|\phi(s)\|_{\Hilbert}^2 
  \,,
\end{equation}
where~$\lambda_{\theta,a}(s)$ is the lowest eigenvalue of~$L_s$.
Then~\eqref{energy} together with~\eqref{spec.bound} implies Gronwall's inequality
\begin{equation}\label{Gronwall}
  \|\phi(s)\|_{\Hilbert}
  \leq \|\phi_0\|_{\Hilbert} \
  \Euler^{-\int_0^s \lambda_{\theta,a}(\tau) \, \der\tau}
\end{equation}
valid for every $s \geq 0$.
From~\eqref{Gronwall} with help of~\eqref{noweight}, \eqref{preserve} 
and the relationship~\eqref{relationship},
we obtain the crucial estimate
\begin{equation}\label{crucial}
  \Gamma_{\theta,a} \geq \lambda_{\theta,a}(\infty)
  := \liminf_{s\to\infty}\lambda_{\theta,a}(s)
  \,.
\end{equation}
We refer to~\cite[Sec.~4.5]{KZ2}, \cite[Sec.~7.10]{KKolb} 
or \cite[Prop.~4.3]{CK}
for more details in similar problems.

\subsection{The asymptotic behaviour}
%
It remains to study the asymptotic behaviour of~$\lambda_{\theta,a}(s)$
as $s \to \infty$. 
If $\theta' \in L^1((0,\infty))$, 
then~$\theta_s'$ converges in the sense of distributions on $(0,\infty)$
to zero as $s \to \infty$.
Hence, it is expectable that~$L_s$ converges, in a suitable sense,
to the operator 
\begin{equation}\label{op.noweight.straight}
  L = -\frac{1}{\rho} \, \partial_\rho \, \rho \, \partial_\rho 
  - \frac{1}{\rho^2} \, \partial_\varphi^2 
  + \frac{\rho^2}{16}
\end{equation}
as $s \to \infty$. The latter should be understood as the operator
associated with the quadratic form~\eqref{l-form}.
The following result confirms this expectation.
\begin{Proposition}\label{Prop.strong}
Let $\theta \in C^1((0,\infty))$ be such that~\eqref{Ass} holds
and $\supp\theta'$ is compact in~$\Real$. 
Then the operator~$L_s$ converges to~$L$ 
in the norm-resolvent sense as $s \to \infty$, \ie,
\begin{equation}\label{nrs}
  \lim_{s \to \infty} 
  \big\|L_s^{-1}-L^{-1}\big\|_{\mathcal{H}\to\mathcal{H}} 
  = 0
  \,.
\end{equation}
\end{Proposition}
\begin{proof}
First of all, we note that~$0$ belongs to the resolvent set of~$L$
and~$L_s$ for all $s \geq 0$.
In fact, by Proposition~\ref{Prop.equivalent}
(proved under hypothesis~\eqref{Ass}),
we have the Poincar\'e-type inequality
\begin{equation}\label{Poincare}
  l_s[\phi] \geq C^{-1} \, l[\phi]
  \geq C^{-1} \lambda_1 \|\phi\|_{\Hilbert}^2
\end{equation}
for every $\phi \in \Dom(l)$, 
where~$\lambda_1$ is the lowest eigenvalue of~$L$,
which is easily seen to be positive 
due to the positivity of the form~\eqref{l-form} 
and the fact that the spectrum of~$L$
is purely discrete (explicit value of~$\lambda_1$
is given in Proposition~\ref{Prop.spec} below).

To prove the uniform convergence~\eqref{nrs},
we shall use an abstract criterion from~\cite[App.]{CK} 
according to which it is enough to show that 
\begin{equation}\label{nrs.criterion}
  \lim_{n \to \infty} 
  \big\|L_{s_n}^{-1}f_n-L^{-1}f\big \|_{\mathcal{H}} 
  = 0
\end{equation}
for every sequence of numbers $\{s_n\}_{n \in \Nat} \subset \Real$
such that $s_n \to \infty$ as $n \to \infty$ 
and every sequence of functions $\{f_n\}_{n \in \Nat} \subset \mathcal{H}$
weakly converging to $f \in \Hilbert$ and such that
$\|f_n\|_{\mathcal{H}}=1$ for all $n \in \Nat$.

We set $\phi_n := L_{s_n}^{-1} f_n$, so that~$\phi_n$
satisfies the weak formulation of the resolvent equation
\begin{equation}\label{re}
  \forall v \in \Dom(l) \,, \qquad
  l_{s_n}(v,\phi_n)
  = (v,f_n)_{\mathcal{H}}
  \,.
\end{equation}
Choosing $v := \phi_n$ for the test function in~\eqref{re},
we have
\begin{equation}\label{resolvent.identity}
  l_{s_n}[\phi_n]
  = (\phi_n,f_n)_{\mathcal{H}}
  \leq \|\phi_n\|_{\mathcal{H}} \|f_n\|_{\mathcal{H}}
  = \|\phi_n\|_{\mathcal{H}}
  \,.
\end{equation}

Recalling~\eqref{Poincare},
we obtain from~\eqref{resolvent.identity}
the uniform bound
\begin{equation}\label{bound0}
  \|\phi_n\|_{\mathcal{H}} \leq \frac{C}{\lambda_1}
  \,.
\end{equation}
At the same time, employing the first inequality in~\eqref{Poincare},
the bounds~\eqref{resolvent.identity} and~\eqref{bound0} yield
\begin{equation}\label{bounds}
  \left\| \partial_\rho \phi_n \right\|_{\mathcal{H}}^2
  \leq \frac{C^2}{\lambda_1}
  \,, \qquad
  \left\|
  \frac{\partial_\varphi \phi_n}{\rho}
  \right\|_{\mathcal{H}}^2
  \leq \frac{C^2}{\lambda_1}
  \,, \qquad
  \big\|\rho\phi_n\big\|_{\mathcal{H}}^2 \leq \frac{16 \, C^2}{\lambda_1}
  \,.
\end{equation}

It follows from~\eqref{bound0} and~\eqref{bounds} 
that $\{\phi_n\}_{n \in \Nat}$ is a bounded sequence in $\Dom(l)$
equipped with the norm $\|\cdot\|_{l}$.
Therefore it is precompact in the weak topology of this space.
Let~$\phi_\infty$ be a weak limit point,
\ie, for an increasing sequence $\{n_j\}_{j\in\Nat} \subset \Nat$
such that $n_j \to \infty$ as $j \to \infty$,
$\{\phi_{n_j}\}_{j\in\Nat}$ converges weakly to~$\phi_\infty$ in~$\Dom(l)$.
Actually, we may assume that the sequence converges strongly in~$\mathcal{H}$
because~$\Dom(l)$ is compactly embedded in~$\mathcal{H}$.
Summing up,
\begin{equation}\label{weak1}
  \phi_{n_j} \xrightarrow[j \to \infty]{w} \phi_\infty
  \quad \mbox{in} \quad \Dom(l)
  \qquad \mbox{and} \qquad
  \phi_{n_j} \xrightarrow[j \to \infty]{} \phi_\infty
  \quad \mbox{in} \quad \mathcal{H}
  \,.
\end{equation}

Now we pass to the limit as $n \to \infty$ in~\eqref{re}.
Taking any test function $v \in C_0^1(U)$
in~\eqref{re}, with~$n$ being replaced by~$n_j$,
and sending~$j$ to infinity, we obtain from~\eqref{weak1} the identity
\begin{equation}\label{identity}
  l(v,\phi_\infty) = (v,f)_\mathcal{H}
  \,.
\end{equation}
In the limit, we have used~\eqref{Ass} 
to get rid of the terms containing~$\theta_{s_n}'$.
More specifically, we write
$$
  | (\partial_\rho v,\theta_{s_n}'\partial_\varphi\phi_n)_\Hilbert |
  \leq \|\rho \, \theta_{s_n}' \partial_\rho v\|_\Hilbert 
  \left\|\frac{\partial_\varphi\phi_n}{\rho}\right\|_\Hilbert 
  \,,
$$
where the second term on the right hand side is bounded due to~\eqref{bounds},
while 
$$
\begin{aligned}
  \|\rho \, \theta_{s_n}' \partial_\rho v\|_\Hilbert^2 
  &= \int_U \rho^2 \Euler^{s_n} \theta'(\Euler^{s_n/2}\rho)^2 \,
  |\partial_\rho v(\rho,\varphi)|^2 \, \rho \, \der \rho \, \der\varphi
  \xrightarrow[n \to \infty]{}
  0
\end{aligned}
$$
by the dominated convergence theorem
using~\eqref{Ass} (to get a dominating function)  
and $\theta'(r)=0$ for all sufficiently large~$r$
(to get a pointwise convergence). 
A similar argument holds for the other terms containing~$\theta_{s_n}'$. 

Since $C_0^1(U)$ is a core of~$l$,
then~\eqref{identity} holds true for all $v\in \Dom(l)$.
We conclude that $\phi_\infty = L^{-1} f$,
for \emph{any} weak limit point of $\{\phi_n\}_{n \in \Nat}$.
From the strong convergence of $\{\phi_{n_j}\}_{j\in\Nat}$,
we eventually conclude with~\eqref{nrs.criterion}.
\end{proof}
%

\subsection{A lower bound to the decay rate}
%
Since the operator~$L$ is naturally decoupled,
its spectrum is easy to find.
\begin{Proposition}\label{Prop.spec}
We have
\begin{equation}\label{set.sum}
  \sigma(L) = 
  \left\{
  n + \frac{1}{2} \left(1+\frac{m}{2a} \right)
  \right\}_{n \in \Nat, m \in \Nat^*}
  \,,
\end{equation}
where we use the convention $0 \in \Nat$
and denote $\Nat^* := \Nat\setminus\{0\}$.
\end{Proposition}
\begin{proof}
We have the direct-sum decomposition
(\cf~\cite[Ex.~X.1.4]{RS2})
$$
  L
  = \bigoplus_{m=1}^\infty
  L_m
  \,, \qquad
  L_m :=
  -\frac{1}{\rho} \, \partial_\rho \, \rho \, \partial_\rho 
  + \frac{\nu_m}{\rho^2} 
  + \frac{\rho^2}{16}
  \,,
$$
where, for each fixed $m \in \Nat^*$,
$L_m$ is an  operator in 
$\sii\big((0,\infty),\rho\,\der\rho\big)$
and
\begin{equation}\label{chi}
  \nu_m := \left(\frac{m}{2 a}\right)^2
  \qquad \mbox{and} \qquad
  \sin(\nu_m\varphi)
  \,,
\end{equation}
with $m \in \Nat^*$,
are respectively the eigenvalues 
and the corresponding eigenfunctions
of the operator $-\partial_\varphi^2$ in $\sii\big((0,2\pi a)\big)$, 
subject to Dirichlet boundary conditions.
Fixing $m\in\Nat^*$,
the eigenvalues and the corresponding eigenfunctions of 
the one-dimensional operator~$L_m$
are respectively given by (\cf~\cite[Prop.~3]{K7})
$$
  n + \frac{1+\sqrt{\nu_m}}{2}
  \qquad \mbox{and} \qquad
  \rho^{\nu_m} \, \Euler^{-\rho^2/8} \, L_n^{\nu_m}(\rho^2/4)
  \,,
$$
where $n \in \Nat$ and 
$L_n^{\mu}$ are the generalised Laguerre polynomials
(see, \eg, \cite[Sec.~6.2]{Andrews-Askey-Roy_2000}).
Summing up, using the decoupled form of~\eqref{op.noweight.straight},
the spectrum of~$L$ is composed of the eigenvalue sum
\begin{equation}\label{set.sum.bis}
  n + \frac{1+\sqrt{\nu_m}}{2} 
  \,, \qquad
  n \in \Nat \,, \ m\in\Nat^* \,,
\end{equation}
associated with the eigenfunctions
\begin{equation}\label{Laguerre}
  \rho^{\nu_m} \, \Euler^{-\rho^2/8} \, L_n^{\nu_m}(\rho^2/4)
  \, \sin(\nu_m\varphi) 
  \,.
\end{equation}
Formula~\eqref{set.sum.bis} leads to the set sum~\eqref{set.sum}.
\end{proof} 

As a consequence of Proposition~\ref{Prop.strong},
the eigenvalues of~$L_s$ converge to the eigenvalues of~$L$ as $s \to \infty$.
In particular, for the lowest eigenvalue we have
\begin{equation}\label{lowest}
  \lambda_{\theta,a}(\infty)
  = \lim_{s\to\infty}\lambda_{\theta,a}(s) = \inf\sigma(L) 
  = \frac{1}{2} + \frac{1}{4a} 
  \,,
\end{equation}
where the second equality follows from Proposition~\ref{Prop.spec}.
Recalling~\eqref{crucial}, we have thus established 
the following result.
\begin{Theorem}\label{Thm.main.lbound}
Let $\theta \in C^1((0,\infty))$ be such that~\eqref{Ass} holds
and $\supp\theta'$ is compact in~$\Real$. 
Then 
\begin{equation}\label{result.lbound}
  \Gamma_{\theta,a} \geq \frac{1}{2} + \frac{1}{4a} 
  \,.
\end{equation}
\end{Theorem}
%

\subsection{An upper bound to the decay rate
-- proof of Theorem~\ref{Thm.main}}
%
In view of Theorem~\ref{Thm.main.lbound},
it remains to show that the lower bound to~$\Gamma_{\theta,a}$ is optimal.
\begin{Theorem}\label{Thm.optimal}
Let $\theta \in C^1((0,\infty))$ be such that~\eqref{Ass} holds 
and $\supp\theta'$ is compact in~$\Real$.
Then 
\begin{equation}\label{result.optimal}
  \Gamma_{\theta,a} \leq \frac{1}{2} + \frac{1}{4a} 
  \,.
\end{equation}
\end{Theorem}
\begin{proof}
By definition~\eqref{rate},
it is enough to find an initial datum $\psi_0 \in \Hilbert_w$
such that the solution of~\eqref{Cauchy.straight} satisfies the inequality
$\|\psi(t)\|_{\Hilbert} \geq c \, (1+t)^{-\lambda_{\theta,a}(\infty)}$
for all $t \geq 0$ with some positive constant~$c$
that may depend on~$\psi_0$.
We choose $\psi_0(r,\varphi):=w(x)^{-1/2} \phi_0(r,\varphi)$, where
$$
  \phi_0(r,\varphi) := r^{\nu_1} \, \Euler^{-|x|^2/8} 
  \, \sin(\nu_1\varphi) 
$$
is the eigenfunction of~$L$
corresponding to the lowest eigenvalue~\eqref{lowest}
(\cf~\eqref{Laguerre},
where we have used the identities $L_0^{\mu}=1$ and $H_0=1$).
Then the function~$\phi$ defined by~\eqref{noweight} and~\eqref{SST} 
solves~\eqref{Cauchy.noweight}, 
where~$\psi$ is the solution of~\eqref{Cauchy.straight}. 
Let~$R$ be such that $\theta'(r)=0$ for all $r \geq R$.
By~\eqref{relationship},
$\theta_s'(\rho)=0$ for all $\rho \geq \Euler^{-s/2} R =: R_s$.
Since the action of~$L_s$ coincides with the action of~$L$
wherever $\theta_s'=0$, we have the explicit solution
$$
  \phi(\rho,\varphi,s) = \Euler^{-s \lambda_{\theta,a}(\infty)} 
  \phi_0(\rho,\varphi)
  \,.
$$
for every $s \geq 0$ and
$
  (\rho,\varphi) \in 
  [R_s,\infty) \times (0,2\pi a) 
$.
Recalling~\eqref{preserve} and~\eqref{noweight}
together with the relationship~\eqref{relationship}, 
we get
$$
\begin{aligned}
  \|\psi(t)\|_{\Hilbert}
  = \|w^{-1/2}\phi(s)\|_{\Hilbert}
  &\geq \|\chi_{R_s} \, w^{-1/2}\phi(s)\|_{\Hilbert}
  \\
  &= \Euler^{-s \lambda_{\theta,a}(\infty)}
  \|\chi_{R_s} \, \psi_0\|_{\Hilbert}
  \\
  & \geq \Euler^{-s \lambda_{\theta,a}(\infty)}
  \|\chi_{R} \,\psi_0\|_{\Hilbert}
  \\
  & =(1+t)^{-\lambda_{\theta,a}(\infty)} \, 
  \|\chi_{R} \,\psi_0\|_{\Hilbert}
\end{aligned}
$$
for all $t \geq 0$,
where~$\chi_r$ denotes the characteristic function 
of the set $(r,\infty) \times (0,2\pi a)$.
\end{proof}

Theorem~\ref{Thm.main} follows as a consequence 
of Theorems~\ref{Thm.main.lbound} and~\ref{Thm.optimal}.

\subsection{From normwise to pointwise bounds
-- proof of Corollary~\ref{Thm.point}}
%
Corollary~\ref{Thm.point} follows as a consequence of 
this more general result.
\begin{Theorem}\label{Thm.point.bis}
Let $\theta \in C^1((0,\infty))$ be such that~\eqref{Ass} holds 
and $\supp\theta'$ is compact in~$\Real$.
For any positive numbers~$\eps$ and~$\delta$, 
there exists a constant $C_{\delta,\eps}$
such that the solution~$u$ of~\eqref{Cauchy} 
with an arbitrary initial datum $u_0 \in \sii_w(\Omega)$ obeys
\begin{equation}\label{result.point.bis}
  \forall t \geq \eps \,, \qquad
  \|u(t)\|_{L^\infty(\Omega)}
  \leq 
  C_{\delta,\eps} \, (1+t-\eps)^{-\Gamma_{\theta,a}+\delta} \,
  \|u_0\|_{\sii_w(\Omega)} 
  \,,
\end{equation}
where~$\Gamma_{\theta,a}$ is given by~\eqref{result}.
\end{Theorem}
\begin{proof}
Using the semigroup property, the solution~$u$ of~\eqref{Cauchy} satisfies
$$
  u(t) = \Euler^{t \Delta_D^\Omega} u_0
  = \Euler^{\eps \Delta_D^\Omega} \Euler^{(t-\eps) \Delta_D^\Omega} u_0
  = \Euler^{\eps \Delta_D^\Omega} u(t-\eps)
$$ 
for every $t \geq \eps > 0$. 
By~\cite[Thm.~2.1.6]{Davies_1989}, the heat kernel $k(x,x',t)$ 
of $\Euler^{t \Delta_D^\Omega}$ is bounded by the heat kernel
in the whole Euclidean space, \ie,  
\begin{equation}\label{Davies}
  0 \leq k(x,x',t) \leq (4\pi t)^{-1} \, 
  \Euler^{-|x-x'|^2/(4 t)}
\end{equation}
for every $t \in (0,\infty)$ and $x,x' \in \Omega$.
Using the Schwarz inequality and~\eqref{Davies}, 
we get
$$
\begin{aligned}
  |u(x,t)|
  &= \left| \int_\Omega k(x,x',\eps) \, u(x',t-\eps) \, \der x' \right|
  \\
  &\leq \|u(t-\eps)\|_{\sii(\Omega)} \
  \sqrt{\int_\Omega k(x,x',\eps)^2  \, \der x'}
  \\
  &\leq \|u(t-\eps)\|_{\sii(\Omega)} \
  (4\pi\eps)^{-1} \, 
  \sqrt{\int_{\Real^d} \Euler^{-|x-x'|^2/(2 \eps)}  \, \der x'}
  \\
  &= \|u(t-\eps)\|_{\sii(\Omega)} \
  (4\pi\eps)^{-1} \, (2\pi\eps)^{1/2}
\end{aligned}  
$$
for every $t \geq \eps > 0$ and $x \in \Omega$.
Denoting 
$
  c_{\eps} := (4\pi\eps)^{-1} \, (2\pi\eps)^{1/2}
$
and using~\eqref{solution.rate}, we eventually obtain 
$$
  \|u(t)\|_{L^\infty(\Omega)}
  \leq c_{\eps} \,  
  C_\delta \, (1+t-\eps)^{-\Gamma_{\theta,a}+\delta} \,
  \|u_0\|_{\sii_w(\Omega)} 
$$
for every $t \geq \eps > 0$ and $\delta>0$.
\end{proof}
%

\section{The Hardy inequality}\label{Sec.end}
%
This section is devoted to a proof of Theorem~\ref{Thm.Hardy}.
We again use the curvilinear coordinates of Section~\ref{Sec.curve}.

\subsection{An immediate Hardy inequality}
Employing~\eqref{h-form}, it is easy to establish~\eqref{Hardy} with $c=0$,
that is, 
\begin{equation}\label{Hardy.immediate}
   \forall  u \in H_0^1(\Omega)
  \,, \qquad
  \int_{\Omega} |\nabla u(x)|^2 \, \der x
  \geq \frac{1}{4 a^2} 
  \int_{\Omega} \frac{|u(x)|^2}{|x|^2} \, \der x 
  \,.
\end{equation}
Indeed, it is enough to estimate the angular part of the gradient in~\eqref{h-form}
by the lowest eigenvalue of the operator $-\partial_\varphi^2$ 
in $\sii((0,2\pi a))$, subject to Dirichlet boundary conditions, 
\ie~to use the Poincar\'e-type inequality (\cf~\eqref{chi}),
\begin{equation}\label{Poincare.elementary}
  \forall  u \in H_0^1((0,2\pi a))
  \,, \qquad
  \int_0^{2 \pi a} |f'(\varphi)|^2 \, \der \varphi
  \geq \frac{1}{4 a^2} 
  \int_0^{2 \pi a} |f(\varphi)|^2  \, \der \varphi
  \,,
\end{equation}
and simply neglect the other term in~\eqref{h-form}.

By a test-function argument,
it is also easy to see that the immediate Hardy inequality~\eqref{Hardy.immediate}
is optimal for straight wedges (\ie~$\theta'=0$), 
in the sense that the operator
$-\Delta_D^{\Omega_0} - (4a^2|x|^2)^{-1} - V(x)$
(where the first sum should be understood in a form sense)
possesses negative eigenvalues for any non-negative non-trivial 
potential $V \in C_0^\infty(\Omega_0)$.
Alternatively, the claim can be found in
\cite[Ex.~11.1]{Devyver-Fraas-Pinchover_2014}
(see also \cite[Ex.~1.4]{Devyver-Pinchover-Psaradakis}),
where the question of optimal Hardy weights for elliptic operators
is treated in a great generality.

The content of our Theorem~\ref{Thm.Hardy} is that 
a positive term can be added on the right hand side of~\eqref{Hardy.immediate}
provided that~$\theta'$ is not identically equal to zero. 
Our approach employs some ideas of~\cite{KZ1}.

\subsection{An improved local Hardy inequality}
First, we establish an improved Hardy inequality,
for which the added term on the right hand side of~\eqref{Hardy.immediate} 
is not positive everywhere in~$\Omega$. 

Given a positive number~$R$, define
$
  U_R := (0,R) \times (0,2\pi a)
$.
Using~\eqref{layer}, $\Omega_R := \mathcal{L}(U_R)$
is a bounded subset of~$\Omega$. 
We set
\begin{equation}\label{lambda}
  \lambda(R,\theta')
  :=
  \inf_{\stackrel[\psi\not=0]{}{\psi \in C_0^1(U)}}
  \frac{\displaystyle \int_{U_R} 
  \left[
  \left|\big(\partial_r-\theta'(r)\partial_\varphi\big)\psi\right|^2
  + \left|\frac{\partial_\varphi \psi}{r}\right|^2
  - \frac{1}{4a^2} \left|\frac{\psi}{r}\right|^2
  \right]
  r \, \der r \, \der\varphi}
  {\displaystyle \int_{U_R} |\psi|^2 \, r \, \der r \, \der\varphi} 
  \,.
\end{equation}
We emphasise that the test functions~$\psi$ 
are restrictions of functions from the whole~$U$,
so that the minimisers of~\eqref{lambda} satisfy 
a Neumann boundary condition on $\{R\} \times (0,2\pi a)$.
The following result follows easily from Proposition~\ref{Prop.domain}
and definition~\eqref{lambda}.
\begin{Theorem}\label{Thm.Hardy.local}
Let $\theta \in C^1((0,\infty))$ be such that~\eqref{Ass} holds. 
We have
\begin{equation}\label{Hardy.local}
  \forall  u \in H_0^1(\Omega)
  \,, \qquad
  \int_{\Omega} |\nabla u(x)|^2 \, \der x
  - \frac{1}{4 a^2} 
  \int_{\Omega} \frac{|u(x)|^2}{|x|^2} \, \der x 
  \geq 
  \lambda(R,\theta') \int_{\Omega_R}|u(x)|^2 \, \der x 
  \,.
\end{equation}
\end{Theorem}

Of course, \eqref{Hardy.local}~represents an improvement 
upon~\eqref{Hardy.immediate} only if the number $\lambda(R,\theta')$
is positive.
In this case we call~\eqref{Hardy.local} a local Hardy inequality.
It turns out that it can be always achieved
provided that~$\Omega$ is non-trivially curved.

\begin{Proposition}\label{Prop.positive}
Let $\theta \in C^1((0,\infty))$ be such that~\eqref{Ass} holds. 
If $\theta' \not= 0$, then there exists a positive number~$R_0$ 
such that
$$
  \lambda(R,\theta') > 0
$$
for all $R \geq R_0$.
\end{Proposition}
\begin{proof}
Let~$R_0$ be any positive number for which~$\theta'$
is not identically equal to zero on $(0,R_0)$.
Then, of course, $\theta'\not=0$ on $(0,R)$ for every $R \geq R_0$. 
Because of the boundedness of~$U_R$, one can show that 
the infimum in~\eqref{lambda} is achieved 
by a function $\psi \in \sii(U_R, r \, \der r \, \der\varphi)$  
satisfying
\begin{equation}\label{integrals}
  \int_{U_R} 
  \left[
  \left|\big(\partial_r-\theta'(r)\partial_\varphi\big)\psi\right|^2
  \right]
  r \, \der r \, \der\varphi
  < \infty
  \qquad \mbox{and} \qquad
  \int_{U_R} 
  \left[
  \left|\frac{\partial_\varphi \psi}{r}\right|^2
  - \frac{1}{4a^2} \left|\frac{\psi}{r}\right|^2
  \right]
  r \, \der r \, \der\varphi
  < \infty
  \,.
\end{equation}
Moreover, by elliptic regularity, $\psi$~is smooth in~$U_R$.
Now let us assume, by contradiction, that $\lambda(R,\theta')=0$.
Then the integrals in~\eqref{integrals} are simultaneously
equal to zero due to~\eqref{Poincare.elementary}.
From the vanishing of the second integral, we obtain
$$
  \psi(r,\varphi) = g(r) \sin(\nu_1\varphi)
  \,,
$$
where~$\nu_1$ is the first angular eigenvalue defined in~\eqref{chi}
and~$g$ is a smooth function.  
Plugging this separated function~$\psi$ 
into the first integral in~\eqref{integrals},
putting it equal to zero and integrating by parts, 
we conclude with the two identities
$$
  \int_0^R |g'(r)|^2 \, r \, \der r = 0
  \qquad \mbox{and} \qquad
  \int_0^R \theta'(r)^2 \, |g(r)|^2 \, r \, \der r = 0
  \,.
$$
It follows that $\theta' = 0$ on $(0,R)$, a contradiction.
\end{proof}

\subsection{An improved global Hardy inequality
-- proof of Theorem~\ref{Thm.Hardy}}
In the next step, we produce from the local Hardy inequality~\eqref{Hardy.local}
the desired inequality~\eqref{Hardy}, with an everywhere positive Hardy weight. 
Here the main ingredient is the following classical one-dimensional Hardy inequality,
which we present without proof (\cf~\cite[Lem.~3.1]{CK}). 
\begin{Lemma}\label{Lem.log}
Let $r_0 > 0$. We have 
$$
  \forall g \in C_0^1((r_0,\infty)) 
  \,, \qquad
  \int_{r_0}^\infty |g'(r)|^2 \, r \, \der r
  \geq \frac{1}{4} \int_{r_0}^\infty 
  \frac{|g(r)|^2}{r^2 \log^2(r/r_0)} \, r \, \der r
  \,.
$$
\end{Lemma}

Now we are in a position to prove Theorem~\ref{Thm.Hardy}.
\begin{proof}[Proof of Theorem~\ref{Thm.Hardy}]
By virtue of Proposition~\ref{Prop.domain},
it is enough to prove
\begin{equation}\label{Hardy.enough}
  \tilde{h}[\psi] :=
  h[\psi] - \frac{1}{4a^2} \left\|\frac{\psi}{r}\right\|_\Hilbert^2
  \geq c \int_U \frac{|\psi|^2}{1+r^2\log^2(r)} \, r \, \der r \, \der\varphi
\end{equation}
for every $\psi \in C_0^1(U)$.
From now on, we fix some $R>0$ such that the support of~$\theta'$ 
lies inside the interval $[0,R]$. 
We note that,
since~$\theta$ is assumed to be $C^1$-smooth up to the boundary point~$0$
and the support of~$\theta'$ is compact in~$\Real$,
the condition~\eqref{Ass} is satisfied and~$\theta'$ is bounded. 

First, we shall apply Lemma~\ref{Lem.log} with $r_0 := R/2$.
Let $\xi : (0,\infty) \to [0,1]$ be a smooth function 
with support disjoint with the interval $[0,r_0]$
and such that $\xi=1$ on $(R,\infty)$.
(We keep the same notation~$\xi$ for the function $\xi \otimes 1$ 
on $(0,\infty) \times (0,2\pi a)$.)
Writing $\psi = \xi\psi + (1-\xi)\psi$
and using Lemma~\ref{Lem.log} with help of Fubini's theorem,
we get
\begin{eqnarray}\label{smear}
\lefteqn{
  \int_{U}
  \frac{|\psi|^2}{1+r^2\log^2(r/r_0)} \, r \, \der r \, \der\varphi
}
  \nonumber \\
  &&\leq
  2 \int_{U}
  \frac{|\xi \psi|^2}{r^2\log^2(r/r_0)} \ r \, \der r \, \der\varphi
  +  2 \int_{U}
  |(1-\xi) \psi|^2 \, r \, \der r \, \der\varphi
  \nonumber \\
  &&\leq
  8 \int_{U}
  |\partial_r(\xi \psi)|^2 \ r \, \der r \, \der\varphi
  +  2 \int_{U_R}
  |\psi|^2 \, r \, \der r \, \der\varphi
  \nonumber \\
  &&\leq
  16\int_{U}
  |\partial_r\psi|^2 \ r \, \der r \, \der\varphi
  + 16 \int_{U_R}
  |\partial_r \xi|^2 \, |\psi|^2 \ r \, \der r \, \der\varphi
  +  2 \int_{U_R}
  |\psi|^2 \, r \, \der r \, \der\varphi
  \nonumber \\
  &&\leq
  16 \int_{U}
  |\partial_r \psi|^2 \ r \, \der r \, \der\varphi
  + \left( 16 \, \|\xi'\|_\infty^2 + 2 \right)
  \int_{U_R}
  |\psi|^2 \, r \, \der r \, \der\varphi
  \,.
\end{eqnarray}
Here $\|\xi'\|_\infty$ is the supremum norm of the derivative of~$\xi$
as a function on~$(0,\infty)$.

Second, to apply~\eqref{smear}, we need to estimate~$\tilde{h}[\psi]$
by the integral involving the radial derivative of~$\psi$.  
It can be achieved by adapting~\eqref{elementary1}
as follows
\begin{align}\label{smear2}
  \tilde{h}[\psi] 
  & \geq \int_U 
  \left[
  \epsilon \, |\partial_r\psi|^2
  + \left|\frac{\partial_\varphi \psi}{r}\right|^2
  \left(
  1 - \frac{\epsilon}{1-\epsilon} \, |r\,\theta'(r)|^2
  \right)
  - \frac{1}{4a^2} \left|\frac{\psi}{r}\right|^2  
  \right]
  r \, \der r \, \der\varphi 
  \nonumber \\
  & \geq \int_U 
  \left[
  \epsilon \, |\partial_r\psi|^2
  -\frac{\epsilon}{1-\epsilon} \, |\theta'(r)|^2 \, \frac{|\psi|^2}{4a^2}
  \right]
  r \, \der r \, \der\varphi 
  \nonumber \\
  & \geq 
  \epsilon \int_U |\partial_r\psi|^2 \, r \, \der r \, \der\varphi 
  -\frac{\epsilon}{1-\epsilon} \, \frac{\|\theta'\|_\infty^2}{4a^2} 
  \int_{U_R} |\psi|^2 \, r \, \der r \, \der\varphi 
  \,.
\end{align}
Here the second inequality is due to~\eqref{Poincare.elementary}
after choosing~$\epsilon$ sufficiently small comparing to the supremum
norm of the function  $r \mapsto r\,\theta'(r)$. 

Finally, by Theorem~\ref{Thm.Hardy.local}, we have
\begin{equation}\label{smear3}
  \tilde{h}[\psi] \geq 
  \lambda(R,\theta') \int_{U_R} |\psi|^2 \, r \, \der r \, \der\varphi 
  \,.
\end{equation}
Combining~\eqref{smear} with~\eqref{smear2} and
interpolating the result with~\eqref{smear3},
we get
\begin{multline*}
  \tilde{h}[\psi] \geq 
  \delta \,\frac{\epsilon}{16}
  \int_U \frac{|\psi|^2}{1+r^2\log^2(r)} \, r \, \der r \, \der\varphi
  \\
  + \left[ 
  (1-\delta) \, \lambda(R,\theta') 
  - \delta \frac{\epsilon}{1-\epsilon} \, \frac{\|\theta'\|_\infty^2}{4a^2} 
  - \delta \, \epsilon \left(\|\xi'\|_\infty^2 + \frac{1}{8}\right)
  \right]
  \int_{U_R} |\psi|^2 \, r \, \der r \, \der\varphi 
\end{multline*}
with any $\delta > 0$.
Since $\lambda(R,\theta')$ is positive due to Proposition~\ref{Prop.positive}
and our hypothesis about~$\theta'$,
we can choose $\delta>0$ in such a way that the square bracket vanishes
and obtain~\eqref{Hardy.enough} with
$$
  c \geq \delta \,\frac{\epsilon}{16} \,
  \inf_{r \in (0,\infty)} \frac{1+r^2\log^2(r)}{1+r^2\log^2(r/r_0)}
  >0
$$
This concludes the proof of Theorem~\ref{Thm.Hardy}.
\end{proof}

\subsection{A conjecture about the heat flow}
Let~$T_D^\Omega$ denote the self-adjoint operator 
associated with the closure of of the quadratic form
\begin{equation}\label{form.new}
  \dot{t}_D^\Omega[u] := \int_{\Omega} |\nabla u(x)|^2 \, \der x
  - \frac{1}{4 a^2} 
  \int_{\Omega} \frac{|u(x)|^2}{|x|^2} \, \der x 
  \,, \qquad
  \Dom(\dot{t}_D^\Omega) := C_0^1(\Omega)
  \,.
\end{equation}
Under the hypotheses of Theorem~\ref{Thm.Hardy} (namely, $\theta' \not= 0$),
$T_D^\Omega$~satisfies a Hardy inequality (\cf~\eqref{Hardy}),
while there is no Hardy inequality for~$T_D^{\Omega_0}$ 
corresponding to a straight wedge.
In the language of~\cite{Pinchover_2013},
$T_D^\Omega$ (with $\theta' \not= 0$) and~$T_D^{\Omega_0}$
are \emph{subcritical} and \emph{critical} operators, respectively.
In accordance with general conjectures stated in 
\cite[Sec.~6]{KZ1} and \cite[Conj.~1]{FKP},
we expect that the heat semigroup associated with $T_D^\Omega$ (with $\theta' \not= 0$)
should decay faster comparing to the heat semigroup associated with $T_D^{\Omega_0}$.

More specifically, to deal with the fact that 
the operators~$T_D^{\Omega}$ and~$T_D^{\Omega_0}$ 
act in different Hilbert spaces,
let us consider instead
$$
  H_D^\Omega := \mathcal{U} \, T_D^{\Omega} \, \mathcal{U}^{-1}
  \qquad\mbox{and}\qquad
  H_D^{\Omega_0} := \mathcal{U}_0 \, T_D^{\Omega_0} \, \mathcal{U}_0^{-1}
  \,,
$$
where the unitary transforms~$\mathcal{U}$ and~$\mathcal{U}_0$
are introduced in Section~\ref{Sec.curve} (recall also~\eqref{diagram}).
The operators~$H_D^{\Omega}$ and~$H_D^{\Omega_0}$ act in 
the same Hilbert space~$\Hilbert$ introduced in~\eqref{unitary}.
Then the general conjecture from~\cite[Sec.~6]{KZ1}
reads as follows:
\begin{Conjecture}
Let $\theta \in C^1([0,\infty))$ 
be such that $\supp\theta'$ is compact in~$\Real$ and $\theta' \not= 0$.
Then there exists a positive function $w: U \to \Real$ such that
$$
  \lim_{t \to \infty}
  \frac{\left\| 
  \, \Euler^{-t \, H_D^\Omega} \, 
  \right\|_{\Hilbert_w \to \Hilbert}}
  {\left\| 
  \, \Euler^{-t \, H_D^{\Omega_0}}
  \right\|_{\Hilbert_w \to \Hilbert}}
  = 0
  \,,
$$
where the weighted space~$\Hilbert_w$ is defined as in~\eqref{analogue.weighted}.
\end{Conjecture}

A similar conjecture can be stated for 
the heat kernels of~$H_D^{\Omega}$ and~$H_D^{\Omega_0}$,
\cf~\cite[Conj.~1]{FKP}.

\subsection*{Acknowledgment}
%
I thank Michiel van den Berg for questions 
that led me to start this work, 
fruitful discussions and useful remarks on a previous 
version of this manuscript.
I am also indebted to Cristian Cazacu and Martin Kolb for valuable comments.
The work was partially supported
by the project RVO61389005 and the GACR grant No.\ 14-06818S.
The award from the \emph{Neuron fund for support of science},
Czech Republic, May 2014, is also acknowledged.

%
\bibliography{bib}
\bibliographystyle{amsplain}

\vfill
\begin{figure}[ht]
\begin{center}
\begin{tabular}{ccc}
\includegraphics[width=0.3\textwidth]{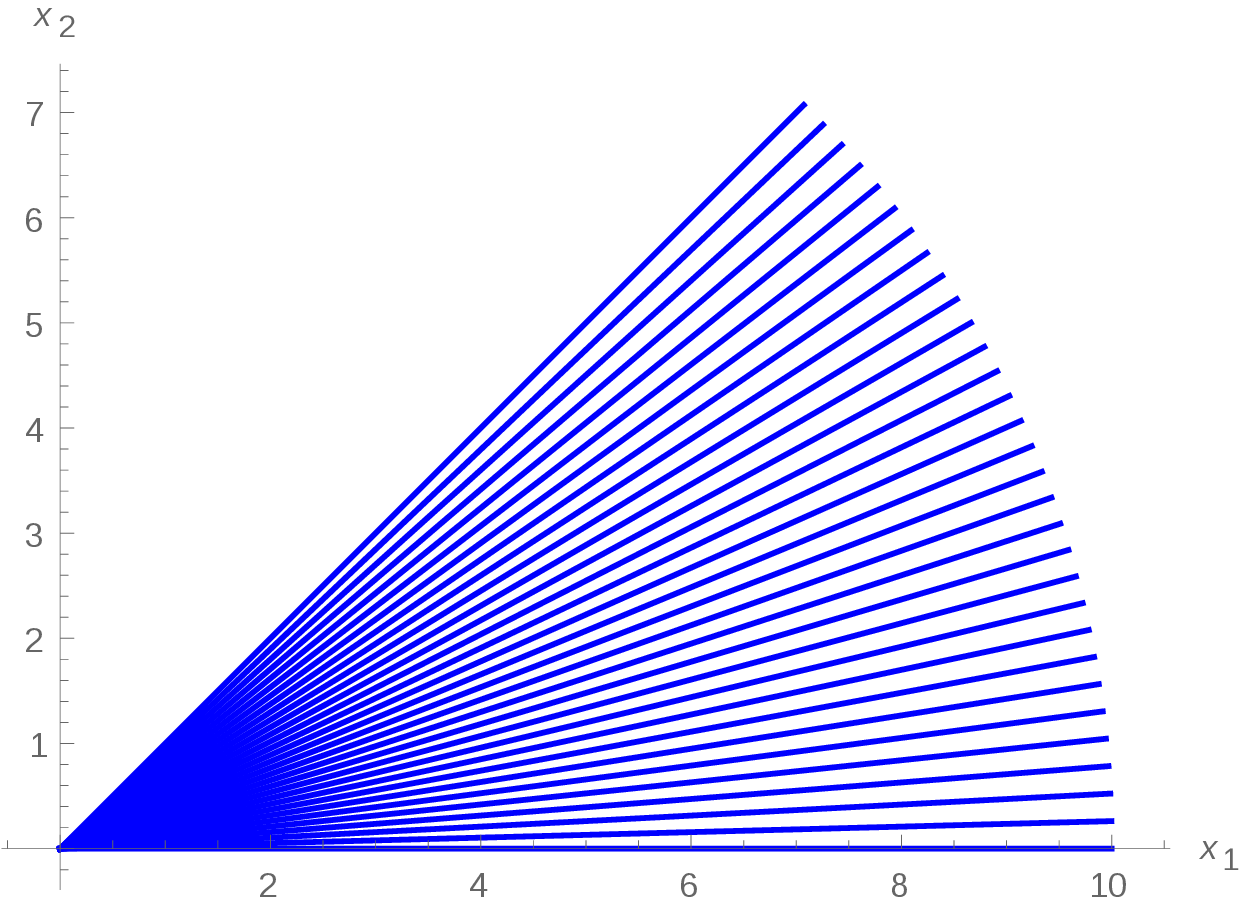}
& \includegraphics[width=0.3\textwidth]{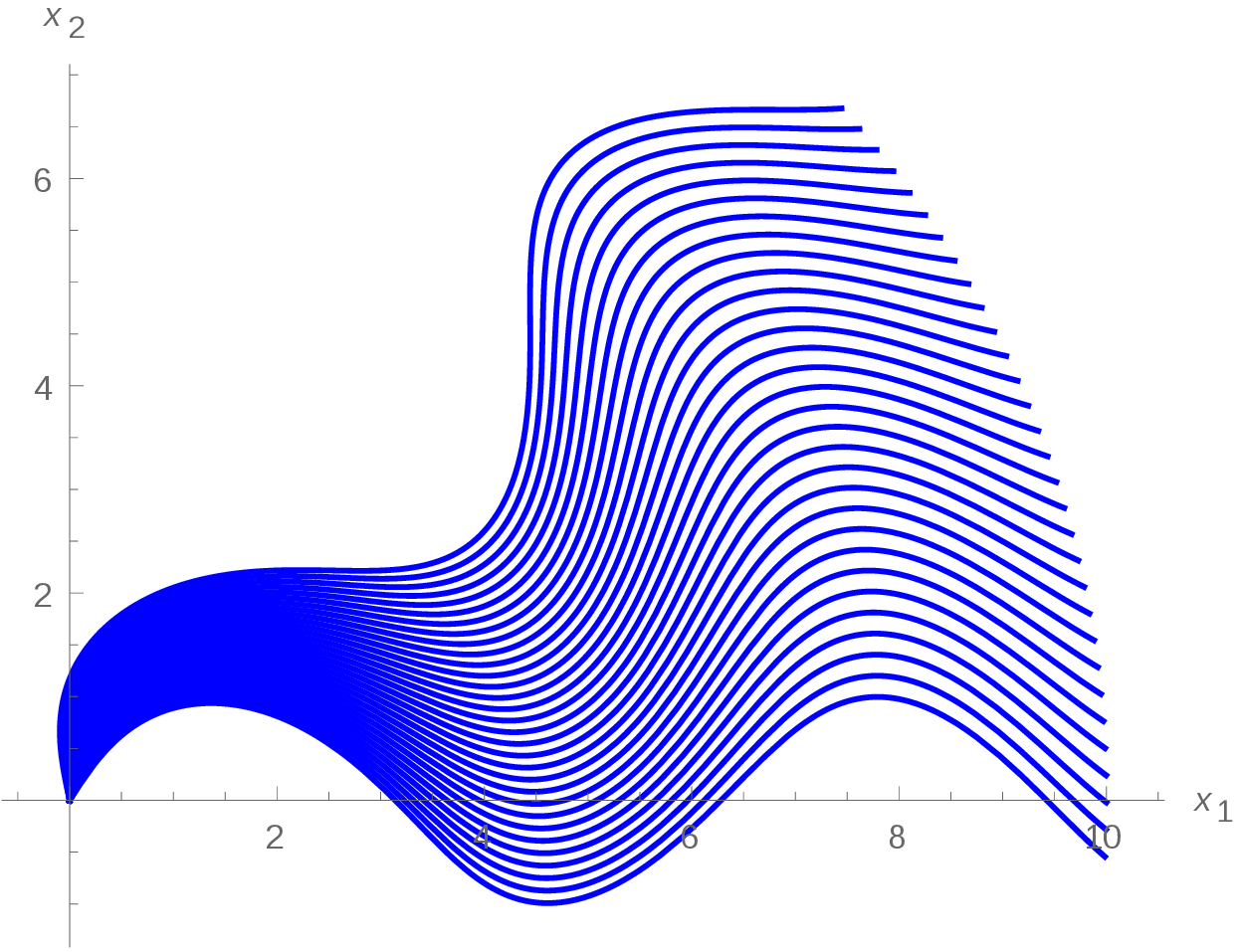}
& \includegraphics[width=0.3\textwidth]{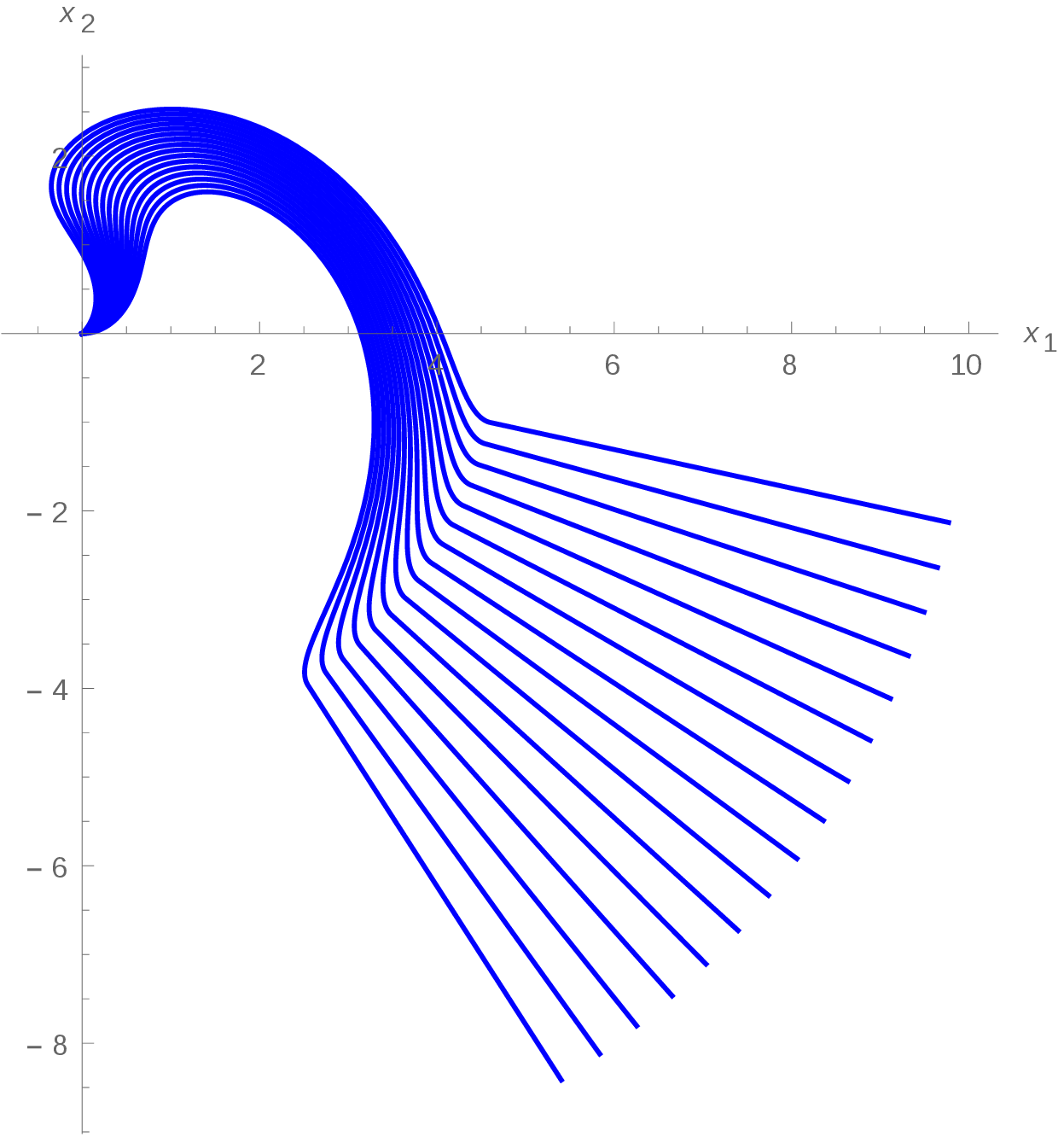}
\\
{\footnotesize (a) \ $\theta(r)=0$}
& {\footnotesize (b) \ $\theta(r)=\sin(r)/r$}
& {\footnotesize (c) \ $\theta(r)=
\begin{cases}
\sin(r) & \mbox{if} \quad r \leq 3\pi/2
\\
-1 & \mbox{if} \quad r > 3\pi/2
\end{cases}
$}
\\
\includegraphics[width=0.3\textwidth]{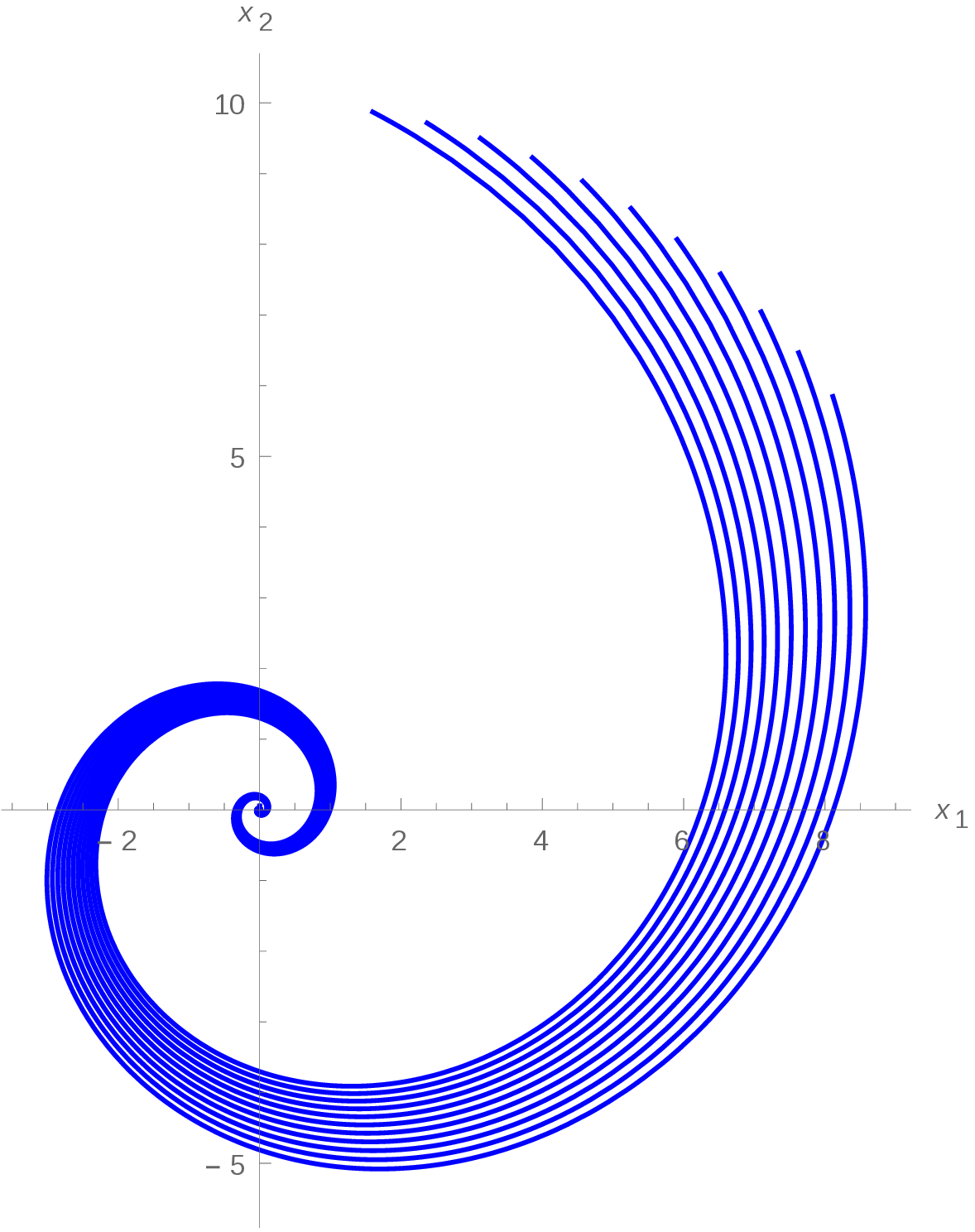}
& \includegraphics[width=0.3\textwidth]{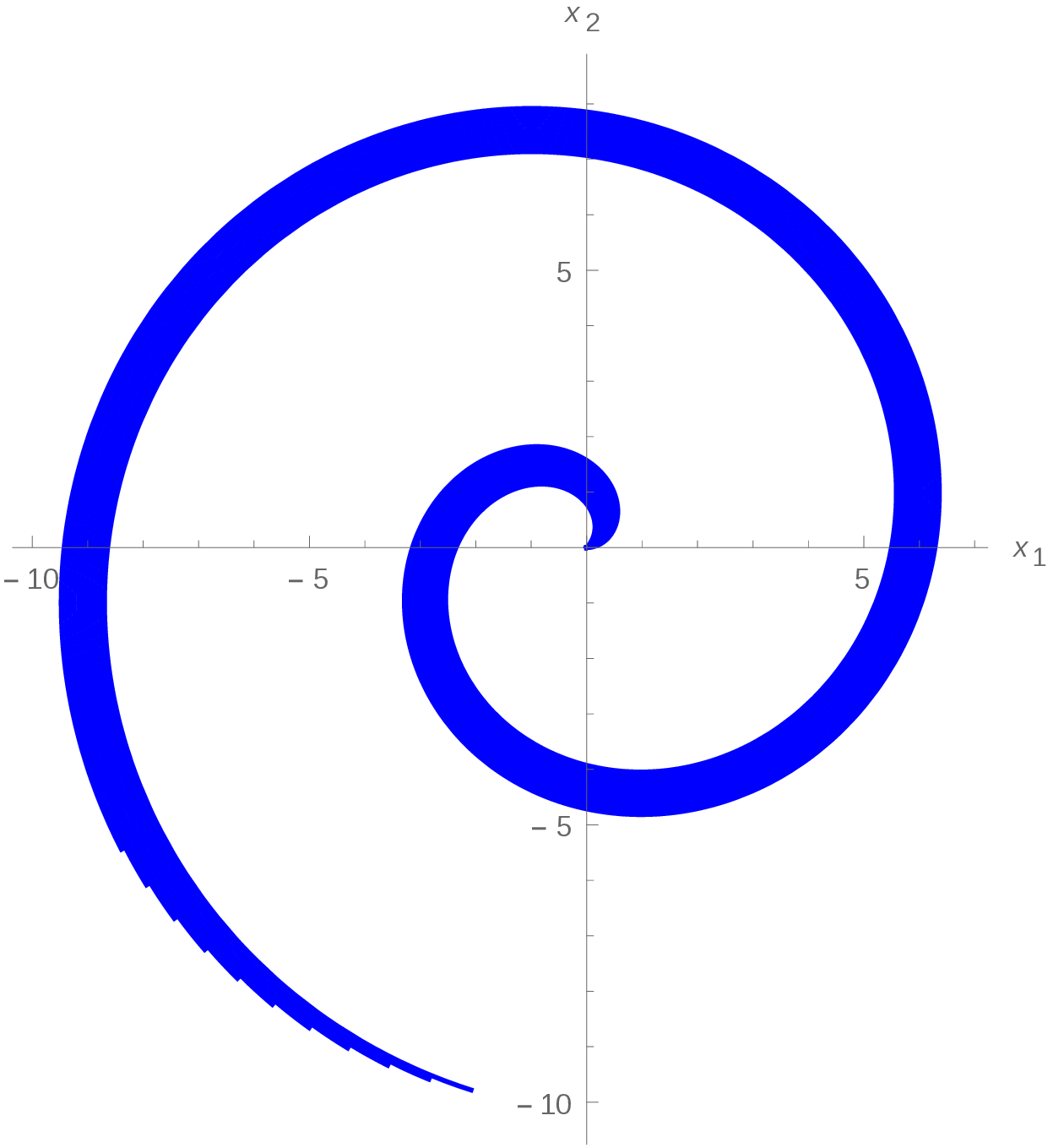}
& \includegraphics[width=0.3\textwidth]{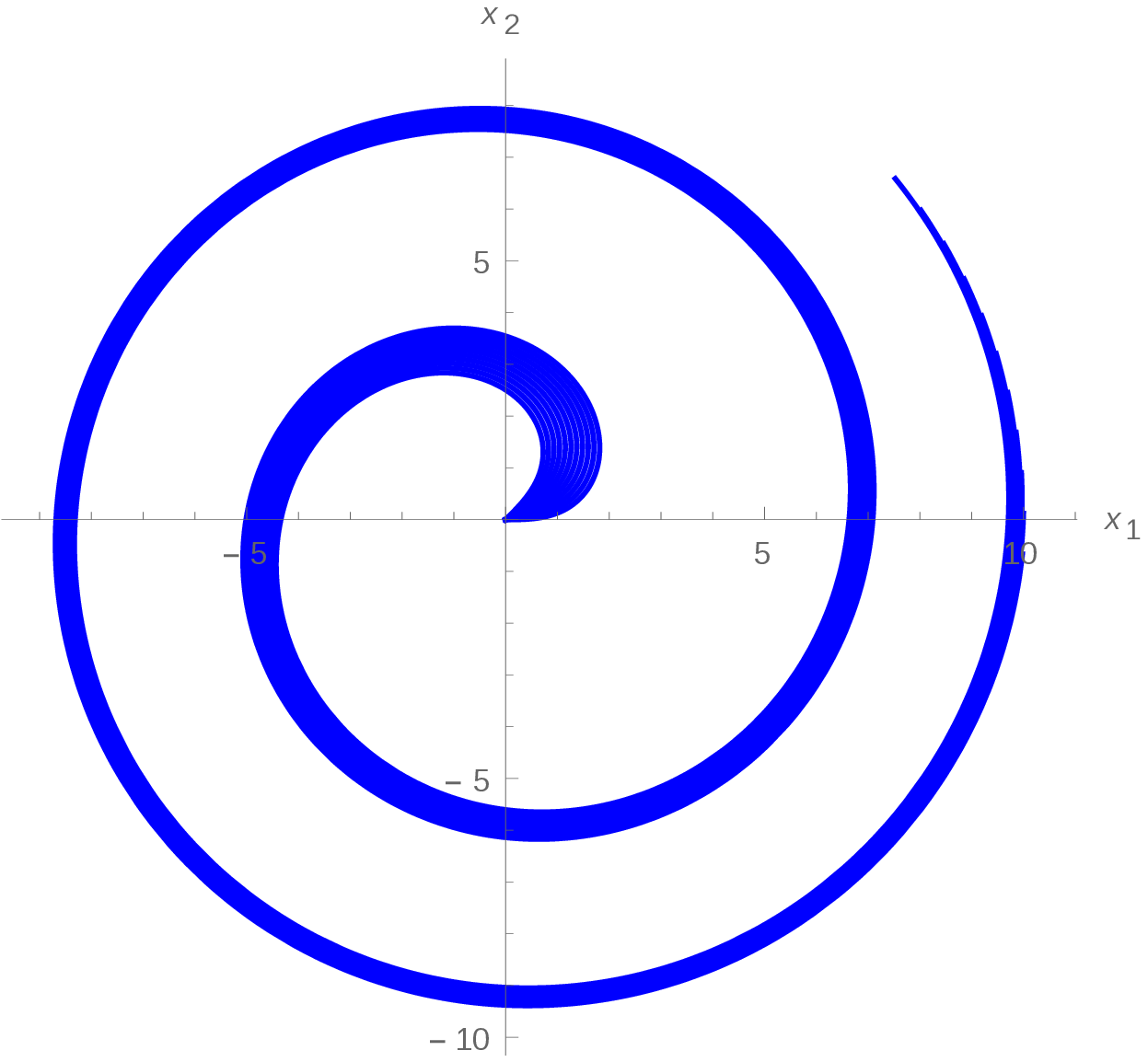}
\\
{\footnotesize (d) \ $\theta(r)=3 \log(r)$}
& {\footnotesize (e) \ $\theta(r)=r$}
& {\footnotesize (f) \ $\theta(r)=r^2/8$}
\end{tabular}
\end{center}
\caption{Examples of curved wedges of opening angle $\pi/4$.
Domains (a)--(d) are all quasi-conical and satisfy
hypotheses~\eqref{local} and~\eqref{Ass} of this paper
(moreover, $\theta'$ is compactly supported for~(a) and~(c)).
Wedges~(e) and~(f) are examples of 
unbounded quasi-cylindrical and quasi-bounded domains, respectively,
that are not considered in this paper.}\label{Fig}
\end{figure}

\end{document}